\documentclass[10pt,a4paper]{article}

\usepackage[utf8]{inputenc}

\usepackage{amsmath}
\usepackage{amsfonts}
\usepackage{amssymb}
\usepackage{mathabx}
\usepackage{bm}
\usepackage{makeidx}
\usepackage{hyperref}
\usepackage[usenames,dvipsnames]{color}
\usepackage{pagecolor,lipsum}

\usepackage{verbatim}

\usepackage{psfrag}
\usepackage{graphicx}

\usepackage{algorithmic}
\usepackage[ruled]{algorithm}

\usepackage{bm}
\usepackage{stmaryrd}

\newtheorem{definition}{Definition}[section]
\newtheorem{lemma}{Lemma}[section]

\newtheorem{assumption}{Assumption}[section]

\newtheorem{remark}{Remark}[section]

\newcommand{\R}{\mathbb R}

\newcommand{\HO}{H}
\newcommand{\HP}{H_D}
\newcommand{\RL}{{\mathcal R}}
\newcommand{\NC}{{k_0}}
\newcommand{\MC}{{k_1}}

\newcommand{\supp}[1]{\operatorname{supp}(#1)}

\newcommand{\id}{{I_d}}

\DeclareMathOperator{\ima}{Im}

\newcommand{\lecnot}{\cite{Dolean:2015:IDDSiam} }

\newcommand{\QED}{\hspace*{\fill}\rule{2.5mm}{2.5mm}}  
\newenvironment{proof}{{\bf Proof\ }}{\QED\\}  

\newcommand{\alerte}[1]{ {\bf\color{red}Alerte :  #1}}

\usepackage{authblk}
\renewcommand{\alerte}[1]{} 
\renewcommand{\id}{I}

\author[1]{F.~Nataf}

\affil[1]{\footnotesize Laboratoire J.L.~Lions, UPMC, CNRS UMR7598, Equipe LJL-INRIA Alpines, frederic.nataf@sorbonne-universite.fr, Paris, France}

\title{Mathematical Analysis of Robustness of Two-Level Domain Decomposition Methods with respect to Inexact Coarse Solves}

\begin{document}

\maketitle 

\setcounter{tocdepth}{5}
\tableofcontents
\pagenumbering{arabic}

\begin{abstract}
Convergence of domain decomposition methods rely heavily on the efficiency of the coarse space used in the second level. The GenEO coarse space has been shown to lead to a robust two-level Schwarz preconditioner which scales well over multiple cores ~\cite{Spillane:2014:ASC,Dolean:2015:IDDSiam}. The robustness is due to its good approximation properties for problems with highly heterogeneous material parameters. It is available in the finite element packages FreeFem++~\cite{Hecht:2012:NDF}, Feel++~\cite{Prudhomme:2006:DSE}, Dune~\cite{Blatt:2016:Distributed} and is implemented as a standalone library in HPDDM~\cite{Jolivet:2014:HPD} and as such is available as well as a PETSc preconditioner. But the coarse component of the preconditioner can ultimately become a bottleneck if the number of subdomains is very large and exact solves are used. It is therefore interesting to consider the effect of inexact coarse solves. In this paper, robustness of GenEO methods is analyzed with respect to inexact coarse solves. Interestingly, the GenEO-2 method introduced in~\cite{haferssas:2016:ASM} has to be modified in order to be able to prove its robustness in this context. 
\end{abstract}

\section{Introduction}
 \label{sec:introduction}

Convergence of domain decomposition methods rely heavily on the efficiency of the coarse space used in the second level, see~\cite{Nicolaides:DCG:1987,Toselli:2005:DDM,Pechstein:2017:UFA} and references therein. {These methods are based on two ingredients: a coarse space (CS) and a correction formula (see e.g.~\cite{Tang:2009:CTL}). The GenEO coarse space introduced in~\cite{Spillane:2014:ASC} has been shown to lead to a robust two-level Schwarz preconditioner which scales well over multiple cores. The robustness is due to its good approximation properties for problems with highly heterogeneous material parameters. This approach is closely related to~\cite{Efendiev:2012:RDD}. We refer to the introduction of~\cite{Spillane:2014:ASC} for more details on the differences and similarities between both approaches. Here we will mainly work with a slight modification of the GenEO CS introduced  in~\cite{Dolean:2015:IDDSiam} for the additive Schwarz method (see e.g.~\cite{Toselli:2005:DDM}) and the GenEO-2 CS introduced in~\cite{haferssas:2016:ASM} for the P.L.~Lions algorithm~\cite{Lions:1990:SAM}. These  variants are easier to implement and in practice have similar performances although they may lead to a larger CS. More details are given in Annex~\ref{sec:annex} where we explain how to adapt the framework of~\cite{Dolean:2015:IDDSiam} to the GenEO CS of~\cite{Spillane:2014:ASC}. }

We focus in this paper on a modification of the coarse component of the correction formula. Indeed, the coarse component of the preconditioner can ultimately become a bottleneck if the number of subdomains is very large and exact solves are used. It is therefore interesting to consider the effect of inexact coarse solves on the robustness. We show that the additive Schwarz method is naturally robust. Interestingly, the GenEO-2 method introduced in~\cite{haferssas:2016:ASM} has to be modified in order to be able to prove its robustness in this context. In the context of domain decomposition methods, the robustness of the BDDC w.r.t. inexact coarse solves has been studied in~\cite{Tu:2007:TLB,Tu:2011:TLB} and in~\cite{Mandel:2008:MBD}. We focus here on GenEO methods. { Compared to works on multilevel methods such as~\cite{Zhang:1992:MSM,Dryja:1996:MSM} which are concerned with Schwarz multilevel methods where the coarse space is obtained by a coarse grid discretisation of the elliptic problem, we explicitly state robustness results of the two level method with respect to inexact coarse solves when the coarse space is obtained by the solution of local generalized eigenvalue problems. Moreover, we are not concerned only with Schwarz methods but also with P.L.~Lions algorithm.}\\

The general framework of our work is the following. {Let $M^{-1}$ be a one level preconditioner further enhanced by a second level correction} based on a rectangular matrix $Z$ whose columns are a basis of a coarse space $V_0$. The coarse space correction is  
 \begin{equation}
 	Z (Z^T\,A\,Z)^{-1} Z^T\,,
 \end{equation}
and the coarse operator is defined by
 \begin{equation}
	 \label{eq:Edef}
 	E := Z^T\,A\,Z\,.
 \end{equation}
Let $M^{-1}$ denote a one-level preconditioner, the {hybrid} two-level method is defined by:
\[
  M^{-1}_2 :=  Z\,E^{-1}\,Z^T +  (\id - Z\,E^{-1}\,Z^T A ) M^{-1} (\id - A Z\,E^{-1}\,Z^T )  \,,
\] 	
see the balancing domain decomposition method by J.~Mandel~\cite{Mandel:1992:BDD}. {This formula also appeared in an  unpublished work by Schnabel \cite{Schnabel:83:QNM}, see~\cite{Gower:2016:SBB} for more details on the connections between these works. }

We consider Geneo methods, where the coarse space $V_0$ spanned by the columns of $Z$ is built from solving generalized eigenvalue problems (GEVP) in the subdomains. Recall that these GEVP solves are purely parallel tasks with no communication involved. This part of the preconditioner setup is not penalizing parallelism. Actually, in strong scaling experiments where the number of degrees of freedom of subdomains is smaller and smaller, the elapsed time taken by these tasks will decrease. Thus, this task scales strongly. On the other hand, as the size of matrix $Z^T\,A\,Z$ typically increases linearly with the number of subdomains, the solving of the corresponding linear systems for instance with a $LU$ factorization becomes a bottleneck in two-level domain decomposition methods. It is therefore interesting to estimate the robustness of the modified two-level method when in \eqref{eq:Edef} the operator $E$ is approximated by some operator $\tilde E$:
\[\boxed{
   \tilde E \simeq E \,,
}\]
since it paves the way to inexact coarse solves or to three or more level methods. {Operator $\tilde E$ may be obtained in many ways: approximate LU factorizations (e.g. ILU(k), ILU-$\epsilon$ or single precision factorization), Sparse Approximate Inverse, Krylov subspace recycling methods, multigrid methods and of course domain decomposition methods. In the latter case, we would have a multilevel method. Note that our results are expressed in terms of the spectral properties of $E\tilde E ^{-1}$ so that an approximation method for which such results exist is preferable.}

More precisely, formula~\eqref{eq:Edef} is modified and the preconditioner we study is defined by:
\[
\tilde  M^{-1}_2 :=   Z\,\tilde E^{-1}\,Z^T +  (\id - Z\,\tilde E^{-1}\,Z^T A ) M^{-1} (\id - A Z\,\tilde E^{-1}\,Z^T )  \,.
\] 	
{and throughout the paper we make 
\begin{assumption}
	\label{ass:EtildeSPD}
	The operator $\tilde E$ is symmetric positive definite. 
\end{assumption}}
   

\section{Basic definitions} 
\label{sec:basic_definitions}

The problem to be solved is defined via a variational formulation on a domain $\Omega\subset\R^d$ for $d\in\mathbb{N}$:
\[
\text{Find }u\in V \text{ such that : }
a_\Omega(u,v) = l(v)\,,\ \ \forall v\in V\,,
\]
where $V$ is a Hilbert space of functions from $\Omega$ with real values. The problem we consider is given through a symmetric positive definite bilinear form $a_\Omega$ that is defined in terms of an integral over any open set $\omega \subset \Omega$. Typical examples are the {heterogeneous diffusion }equation ($\mathbf{K}$ is a diffusion tensor)
\[
a_{\omega}(u,v) :=  \int_\omega \mathbf{K}\, \nabla u \cdot \nabla v \, dx\,,
\]
or the elasticity system ($\boldsymbol{C}$ is the fourth-order stiffness tensor and $\boldsymbol{\varepsilon}(\boldsymbol{u})$ is the strain tensor of a displacement field $\boldsymbol{u}$):
\[
a_{\omega}(\boldsymbol{u},\,\boldsymbol{v}) := \int_\omega \boldsymbol{C} : \boldsymbol{\varepsilon}(\boldsymbol{u}) : \boldsymbol{\varepsilon}(\boldsymbol{v})\, dx \,.
\]
The problem is discretized by a finite element method. Let ${\mathcal N}$ denote the set of degrees of freedom and $(\phi_k)_{k\in {\mathcal N}}$ be a finite element basis on a mesh ${\mathcal{T}}_h$. Let $A\in \R^{\# {\mathcal N}\times\# {\mathcal N}}$ be the associated finite element matrix, $A_{kl}:=a_\Omega(\phi_l,\phi_k)$, $k,l\in {\mathcal N}$. For some given right hand side $\mathbf{F}\in \R^{\#{\mathcal N}}$, we have to solve a linear system in $\mathbf{U}$ of the form
\[
A \mathbf{U} = \mathbf{F}\,.
\]
Domain $\Omega$ is decomposed into $N$ (overlapping or non overlapping) subdomains $(\Omega_i)_{1\le i\le N}$ so that all subdomains are a union of cells of the mesh ${\mathcal{T}}_h$.  This decomposition induces a natural decomposition of the set of indices ${\mathcal N}$ into $N$ subsets of indices $({\mathcal N}_i)_{1\le i\le N}$:
\begin{equation}
	\label{eq:thCG:Ni}
	\mathcal{N}_i := \{ k\in \mathcal{N}\ |\ meas(\supp{\phi_k} \cap \Omega_i) > 0\}\,,\ 1\le i \le N.	
\end{equation}
For all $1\le i\le N$, let $R_i$ be the restriction
	matrix from $\R^{\#{\mathcal N}}$ to the subset $R^{\#{\mathcal
	  N}_i}$ and $D_i$ be a diagonal matrix of size $\#
	{\mathcal N}_i \times \# {\mathcal N}_i$, so that we have a partition of unity at the algebraic level, 
\begin{equation}
	\label{eq:POU}
 \sum_{i=1}^N R_i^T D_i R_i = \id\,,
\end{equation}
where $\id\in\R^{\# {\mathcal N}\times\# {\mathcal N}}$ is the identity matrix.\\

We also define for all subdomains $1\le j\le N$,  $\widetilde A^j$, the $\# {\mathcal N}_j\times \# {\mathcal N}_j$ matrix defined by
\begin{equation}
	  \label{eq:thCG:AtildeFE}
{ {\mathbf{V}}_j^T \widetilde A^j {\mathbf{U}}_j := a_{\Omega_j}\left(\sum_{l\in  {\mathcal N}_j}\mathbf{U}_{jl} \phi_l,\, \sum_{l\in  {\mathcal N}_j}{\mathbf{V}}_{jl} \phi_l\right)\,,\ \  {\mathbf{U}}_j,\, {\mathbf{V}}_j\in \R^{ {\mathcal N}_j}  \,.}
\end{equation}
When the bilinear form $a$ results from the variational solve of a
Laplace problem, the previous matrix corresponds to the discretization
of local Neumann boundary value problems. For this reason we will call it ``Neumann'' matrix even in a more general setting. 

We also make use of two numbers $\NC$ and $\MC$ related to the domain decomposition. Let 
\begin{equation}
	\label{eq:thCG:kzero}
\NC  := \max_{1\le i \le N} \# \left\{ j\ |\   R_j A\,R_i^T\neq 0 \right\}	
\end{equation}
be the maximum multiplicity of the interaction between
subdomains plus one. Let $\MC$ be the maximal multiplicity of subdomains intersection, i.e. the largest integer $m$ such that there exists $m$ different subdomains whose intersection has a non zero measure.\\

   Let $\tilde P_0$ be defined as:
   \begin{equation}
   	\label{eq:potilde}
   \boxed{
   	\tilde P_0:=Z \tilde E^{-1} Z^T A\,,
   }
   \end{equation}
   the operator $\tilde P_0$ is thus an approximation to the $A$-orthogonal projection on $V_0$
   \[\boxed{
   	      P_0:=Z E^{-1} Z^T A
   }\]
   which corresponds to an exact coarse solve.

Note that although $\tilde P_0$ is not a projection it has the same kernel and range as $P_0$:
\begin{lemma}
	\label{th:kerimpotilde}
We have
\[	
	\ker P_0 = \ker \tilde P_0 = V_0^{A\perp} \text{\ \  and \ \ } \ima P_0 = \ima \tilde P_0 = V_0\,,
\]
where $V_0^{A\perp}$ is the vector space $A$-orthogonal to $V_0$, that is when $\R^{\#{\mathcal N}}$ is endowed with the scalar product induced by $A$:
$(x\,,\,y)_A := (x\,,\,Ay)$. 
\end{lemma}
\begin{proof}
First note that the kernel of $\tilde P_0$ contains $\ker Z^TA$. On the other hand, we have:
\[
  \tilde P_0\,x = Z \tilde E^{-1} Z^T A x = 0 \Rightarrow (Z \tilde E^{-1} Z^T A x , A x) = (\tilde E^{-1} Z^T A x , Z^T A x) = 0\,.
\]
Since $\tilde E$ is SPD, it means that $Z^TAx=0$, that is $x\in \ker Z^TA$. We have thus $\ker \tilde P_0 = \ker Z^TA$. Note that
\[
   Z^T A x = 0 \Leftrightarrow \forall y\,\ (Ax\,,\,Z y) = 0 \Leftrightarrow x\in V_0^{A\perp}\,.
\]
As for the image of $\tilde P_0$, since the last operation in its definition is the multiplication by the matrix $Z$ we have $\ima P_0 \subset V_0$. Conversely, let $y\in V_0$, there exists $\beta$ such that $y=Z\beta$. It is easy to check that $y=\tilde P_0 (Z\,(ZAZ)^{-1}\tilde E\, \beta)$. Thus, $\ima \tilde P_0 = V_0$. \\ 
The same arguments hold if $\tilde E$ is replaced by $E$. Thus, $\tilde P_0$ and  $P_0$ have the same kernel and image.	
\end{proof}


\section{Inexact Coarse Solves for GenEO} 
\label{sub:inexact_cs_for_geneo}
The GenEO coarse space was introduced in~\cite{Spillane:2014:ASC} and its slight modification is defined as follows, see~\cite{Dolean:2015:IDDSiam}:

\begin{definition}[Generalized Eigenvalue Problem for GenEO]
	\label{th:tauthresholdgeneo}
	For each subdomain $1\le j\le N$, we introduce the generalized eigenvalue problem
	\begin{equation}
	  \label{eq:eigDADtildeA}
	{\begin{array}{c}
	\mbox{ Find }(\mathbf{V}_{jk},\tau_{jk})\in
	\R^{\# {\mathcal N}_j}\setminus \{0\} \times \R
	\mbox{ such that}\\
	D_j R_j\,A\,R_j^T D_j  \mathbf{V}_{jk} = \tau_{jk} \widetilde A^j  \mathbf{V}_{jk} 
	\, \,.
	\end{array}}
	\end{equation}
	Let $\tau>0$ be a user-defined threshold, we define $V_{geneo}^\tau\subset\R^{\# {\mathcal N}}$ as the vector space spanned by the family of vectors $(R_j^T D_j {\mathbf V}_{jk})_{\tau_{jk} > \tau\,, 1\le j\le N}$ corresponding to eigenvalues larger than $\tau$. 
\end{definition}
Let $\tilde\pi_j$ be the projection from $\R^{\# {\mathcal N}_j}$ on $\text{Span}\{ \mathbf{V}_{jk} |\, \tau_{jk} > \tau \}$ parallel to $\text{Span}\{ \mathbf{V}_{jk} |\, \tau_{jk} \le \tau \}$. 

In this section, $Z$ denotes a rectangular matrix whose columns are a basis of the coarse space $V_{geneo}^\tau$ defined in Definition~\eqref{th:tauthresholdgeneo}. The dimension of $Z$ is $\#{\mathcal N} \times \#{\mathcal N}_0$. The GenEO preconditioner with inexact coarse solve reads:
\begin{equation}
	\label{eq:GenEOACS}
  M^{-1}_{GenEOACS} :=  Z\, \tilde E^{-1}\,Z^T +   (\id - \tilde P_0)\, (\sum_{i=1}^N R_i^T\, (R_i\,A\,R_i^T)^{-1}\,R_i)\, (\id - \tilde P_0^T)\,. 	
\end{equation}

{
The study the spectrum of $M^{-1}_{GenEOACS}A$ is based on the Fictitious Space lemma which is recalled here, see \cite{Nepomnyaschikh:1991:MTT} for the original paper and \cite{Griebel:1995:ATA} for a modern presentation. 
 \begin{lemma}[Fictitious Space Lemma, Nepomnyaschikh 1991]
   \label{th:fictitiousSpaceLemma}
   
   Let $\HO$ and $\HP$ be two Hilbert spaces, with the scalar products
   denoted by $(\cdot,\cdot)$ and $(\cdot,\cdot)_D$. Let the symmetric positive bilinear forms $a\,:\,\HO \times \HO \rightarrow \R$ and $b\,:\,\HP \times \HP \rightarrow \R$, generated by the s.p.d. operators $A\,:\,\HO \rightarrow \HO$ and $B\,:\,\HP \rightarrow \HP$, respectively (i.e. $(Au,v)=a(u,v)$ for all $u,v\in\HO$ and $(Bu_D,v_D)_D=b(u_D,v_D)$ for all $u_D,v_D\in\HP$). Suppose that there exists a linear operator $\RL\,:\,\HP \rightarrow \HO$ that satisfies the following three assumptions:
\begin{itemize}
	\item[(i)] $\RL$ is surjective.
	\item[(ii)] Continuity of $\RL$: there exists a positive constant $c_R$ such that 
	\begin{equation}\label{eq:cr}
	a(\RL u_D,\RL u_D) \le c_R\cdot  b(u_D,u_D)\ \ \forall u_D\in \HP\,.
	\end{equation}
	\item[(iii)] Stable decomposition: there exists a positive constant $c_T$ such that for all $u\in\HO$ there exists $u_D\in \HP$ with $\RL u_D=u$ and 
	\begin{equation}\label{eq:ct}
	c_T\cdot b(u_D,u_D) \le a(\RL u_D,\RL u_D)=a(u,u)\,.
	\end{equation}	
\end{itemize}
We introduce the adjoint operator $\RL^*\,:\,\HO\rightarrow \HP$ by
$(\RL u_D,\,u) = (u_D,\,\RL^* u)_D$ for all $u_D\in\HP$ and
$u\in\HO$.\\
 Then, we have the following spectral estimate
\begin{equation}\label{eq:fictestimate}
c_T\cdot  a(u,u) \le a\left(\RL B^{-1} \RL^* A u,\,u\right) \le c_R\cdot  a(u,u)\,,\ \ \forall u\in\HO\,
\end{equation}
which proves that the eigenvalues of operator $\RL B^{-1} \RL^* A$ are bounded from below by $c_T$ and from above by $c_R$. 
 \end{lemma}
Loosely speaking, the first assumption corresponds to equation~(2.3), page 36 of \cite{Toselli:2005:DDM} where the global Hilbert space is assumed to satisfy a decomposition into subspaces. The second assumption is related to Assumptions~2.3 and 2.4,  page~40 of \cite{Toselli:2005:DDM}. The third assumption corresponds to the Stable decomposition Assumption~2.2 page~40 of \cite{Toselli:2005:DDM}.}\\

In order to apply this lemma to the preconditioned operator $M^{-1}_{GenEOACS}\,A$, we introduce Hilbert spaces $\HO$ and $\HP$ as follows:
\[
	\HO := \R^{\#{\mathcal N}}
\]
endowed with the bilinear form $a({\mathbf U},{\mathbf U}) := (A\,{\mathbf U} , {\mathbf U})$ and 
\[
	\HP := \R^{\#{\mathcal N}_0} \times \Pi_{i=1}^N\R^{\#{\mathcal N}_i}
\]
endowed with the following bilinear form
\begin{equation}
  \label{eq:bHPapprox}
  \begin{array}{rcl}
  \tilde b: \HP \times \HP &\longrightarrow &\R \\    
\phantom{b: }(\,({\mathbf U}_0\,,\,({\mathbf U}_i)_{1\le i \le N}), \,({\mathbf V}_0\,,\,({\mathbf V}_i)_{1\le i \le N})\, ) &\longmapsto& (\tilde E {\mathbf U}_0,\,{\mathbf V}_0) + (R_i\,A\,R_i^T\,{\mathbf U}_i,\,{\mathbf V}_i)\,.
  \end{array}
\end{equation}
We denote by $\tilde B: \HP \rightarrow \HP$ the operator such that $(\tilde Bu_D,v_D)_D=\tilde b(u_D,v_D)$ for all $u_D,v_D\in\HP$.\\ 
Let  $\widetilde\RL: \HP \longrightarrow \HO$ be defined by
\begin{equation}
  \label{eq:RL} 
\widetilde\RL({\cal U}):= Z\, \mathbf{U}_0 +  (\id- \tilde P_0 )\,\sum_{i=1}^N R_i^T {\mathbf U}_i\,,
\end{equation}
{where ${\cal U}:=(\mathbf{U}_0 \,,\,({\mathbf U}_i)_{1\le i\le N})$.} Recall that if we had used an exact coarse space solve, we would have introduced:
\begin{equation}
  \label{eq:RLexact} 
\RL({\cal U}):= Z\, {\mathbf U}_0 +  (\id- P_0 )\,\sum_{i=1}^N R_i^T {\mathbf U}_i\,.
\end{equation}
Note that we have 
\[
   \widetilde\RL({\cal U}) = \RL({\cal U}) + (P_0-\tilde P_0) \,\sum_{i=1}^N R_i^T {\mathbf U}_i\,.
\]
It can be checked that $M^{-1}_{GenEOACS} = \widetilde\RL\, \widetilde B^{-1} \,\widetilde\RL^T$, see \eqref{eq:GenEOACS}. In order to apply the fictitious space Lemma~\ref{th:fictitiousSpaceLemma}, three assumptions have to be checked. \\

$\bullet$ $\widetilde\RL$ is onto.\\
Let $\mathbf{U}\in \HO$, we have $\mathbf{U} = \tilde P_0 \mathbf{U} + (\id- \tilde P_0  )\,\mathbf{U}$. By Lemma~\ref{th:kerimpotilde}, $\tilde P_0 \mathbf{U}\in V_0$ so that there exists $\mathbf{U}_0\in \R^{\#{\mathcal N}_0}$ such that $\tilde P_0 \mathbf{U} = Z \mathbf{U}_0$. {Owing to~\eqref{eq:POU}, }we have
\[
  \mathbf{U} = Z \mathbf{U}_0 + (\id- \tilde P_0  )\,\sum_{i=1}^N R_i^T D_i R_i \mathbf{U} = \widetilde\RL(\mathbf{U}_0, (D_i R_i \mathbf{U})_{1\leq i\leq N} ))
\]

$\bullet$ Continuity of $\widetilde\RL$\\
We have to estimate a constant $c_R$ such that for all ${\mathcal U}=({\mathbf U}_0\,,\,({\mathbf U}_i)_{1\le i \le N})\in \HO$, we have:
\[
 a(\widetilde\RL ({\mathcal U}),\widetilde\RL ({\mathcal U})) \le c_R\, \tilde b({\mathcal U}\,,\,{\mathcal U})\,.
\]
Let $\delta$ be some positive number. Using that the image of $P_0-\tilde P_0$ is $a$-orthogonal to the image of $\id- P_0$, Cauchy-Schwarz inequality and the $a$-orthogonality of the projection $\id-P_0$, we have: 
\[
\begin{array}{lcl}
a(\widetilde\RL ({\mathcal U}),\widetilde\RL ({\mathcal U}))
&=& \| \RL ({\mathcal U}) + (P_0-\tilde P_0) \sum_{i=1}^N R_i^T \mathbf{U}_i   \|_A^2\\[.9em]
&=& \| \RL ({\mathcal U})  \|_A^2
+ 2\, a( Z\, \mathbf{U}_0 + (\id-P_0) \sum_{i=1}^N R_i^T \mathbf{U}_i\, ,\, (P_0-\tilde P_0) \sum_{i=1}^N R_i^T \mathbf{U}_i  )\\[.4em]
& &+ \| (P_0-\tilde P_0) \sum_{i=1}^N R_i^T \mathbf{U}_i   \|_A^2\\[.9em]
&=& \| \RL ({\mathcal U})  \|_A^2
+ 2\, a( Z\, \mathbf{U}_0\, ,\, (P_0-\tilde P_0) \sum_{i=1}^N R_i^T \mathbf{U}_i  )\\[.4em]
& &+ \| (P_0-\tilde P_0) \sum_{i=1}^N R_i^T \mathbf{U}_i   \|_A^2\\[.9em]
&\leq& \| \RL ({\mathcal U})  \|_A^2
+ \delta \|Z\,\mathbf{U}_0\|_A^2 + \frac{1}{\delta}\, \|(P_0-\tilde P_0) \sum_{i=1}^N R_i^T \mathbf{U}_i \|_A^2\\[.4em]
& &+ \| (P_0-\tilde P_0) \sum_{i=1}^N R_i^T \mathbf{U}_i   \|_A^2\\[.9em]
&\leq& \|Z\,\mathbf{U}_0\|_A^2 + \| \sum_{i=1}^N R_i^T \mathbf{U}_i   \|_A^2
+ \delta \|Z\,\mathbf{U}_0\|_A^2 \\[.4em]
& &+ (1+\frac{1}{\delta})\, \|(P_0-\tilde P_0) \sum_{i=1}^N R_i^T \mathbf{U}_i \|_A^2\\[.9em]
&\leq& (1+\delta) \|Z\,\mathbf{U}_0\|_A^2 + (1+\|P_0-\tilde P_0\|_A^2(1+\frac{1}{\delta})) \| \sum_{i=1}^N R_i^T \mathbf{U}_i   \|_A^2\\[.9em]
&\leq& (1+\delta) \lambda_{max}(E\tilde E^{-1}) (\tilde E\,\mathbf{U}_0 , \mathbf{U}_0) + (1+\|P_0-\tilde P_0\|_A^2(1+\frac{1}{\delta}))\,k_0 \sum_{i=1}^N\| R_i^T \mathbf{U}_i   \|_A^2\\[.9em]
&\leq& \max\left(\, (1+\delta)\,\lambda_{max}(E\tilde E^{-1}) ,\, [1+\|P_0-\tilde P_0\|_A^2(1+\frac{1}{\delta})]\,k_0\,\right) \,\,\tilde b({\mathcal U}\,,\,{\mathcal U})\,.
\end{array}
\]
It is possible to minimize over $\delta$ the factor in front of $\tilde b({\mathcal U}\,,\,{\mathcal U})$  using the 
\begin{lemma}
	\label{th:optannexe}
	Let $c,d,\alpha$ and $\beta$ be positive constant, we have
\[
  \min_{\delta>0} \max (c+\alpha \delta , d + \beta \delta^{-1}) = \frac{d+c+\sqrt{(d-c)^2+4\alpha\beta}}{2}\,.
\]
\end{lemma}
\begin{proof}
	The optimal value for $\delta$ corresponds to the equality $c+\alpha \delta = d + \beta \delta^{-1}$.
\end{proof}
 Let
\begin{equation}
	\label{eq:PmPtildeA}
	\epsilon_A := \|P_0-\tilde P_0\|_A\, = \| Z^T (\, (Z^T\,A\,Z)^{-1} - \tilde E^{-1}) Z^T A \|_A,	
\end{equation}
the formula of Lemma~\ref{th:optannexe} yields
\begin{equation}
	\label{eq:crgeneoACS}
c_R := \frac{k_0(1+\epsilon_A^2)+\lambda_{max}(E\tilde E^{-1})+\sqrt{ (k_0(1+\epsilon_A^2)-\lambda_{max}(E\tilde E^{-1}))^2 + 4 \lambda_{max}(E\tilde E^{-1}) k_0 \epsilon_A^2 }  }{2}\,.	
\end{equation}
Actually, $\epsilon_A$ can be expressed in term of the minimal eigenvalue of $E\tilde E^{-1}$. 
\begin{lemma}
	Other formula for $\epsilon_A$:\\
\[
\epsilon_A = \sup_{{\mathbf U}_0\in \R^{\#{\mathcal N}_0}} \frac{(E(E^{-1}-\tilde E^{-1})E {\mathbf U}_0 , {\mathbf U}_0 )}{(E {\mathbf U}_0,{\mathbf U}_0)} = \max (|1-\lambda_{min}(E\tilde E^{-1})|\,,\,|1-\lambda_{max}(E\tilde E^{-1})| )\,.
\]
\end{lemma}
\begin{proof}
Since $P_0-\tilde P_0$ is $A$-symmetric, its norm is also given by 
\[
    \epsilon_A = \sup_{\mathbf U} \frac{|((P_0-\tilde P_0){\mathbf U}\,,\,{\mathbf U})_A|}{\|{\mathbf U}\|^2_A} 
\]
We can go further by using the fact that $P_0$ is a $A$-orthogonal and that $P_0$ and $\tilde P_0$ have the same kernels and images:
\[
\begin{array}{rcl}
\displaystyle	\epsilon_A &=& \displaystyle \sup_{\mathbf U}\displaystyle \frac{|((P_0-\tilde P_0)(P_0\,{\mathbf U}+(\id-P_0)\,{\mathbf U})\,,\,P_0\,{\mathbf U}+(\id-P_0){\mathbf U})_A|}{\|P_0 {\mathbf U}\|^2_A + \|(\id - P_0) {\mathbf U}\|^2_A} \\[.8em]
\displaystyle	&=& \displaystyle\sup_{\mathbf U} \frac{|((P_0-\tilde P_0)P_0\,{\mathbf U}\,,\,P_0\,{\mathbf U})_A|}{\|P_0 {\mathbf U}\|^2_A + \|(\id - P_0) {\mathbf U}\|^2_A} 
= \sup_{\mathbf U} \frac{|((P_0-\tilde P_0)P_0\,{\mathbf U}\,,\,P_0\,{\mathbf U})_A|}{\|P_0 {\mathbf U}\|^2_A } \\[.8em]
&=& \displaystyle\sup_{\mathbf U\in V_0} \frac{|((P_0-\tilde P_0)\,{\mathbf U}\,,\,{\mathbf U})_A|}{\|{\mathbf U}\|^2_A } = \sup_{{\mathbf U}_0\in \R^{\#{\mathcal N}_0}} \frac{|(E(E^{-1}-\tilde E^{-1})E {\mathbf U}_0 , {\mathbf U}_0 )|}{(E {\mathbf U}_0,{\mathbf U}_0)}\\
&=& \displaystyle \sup_{{\mathbf U}_0\in \R^{\#{\mathcal N}_0}} |1 -  \frac{(\tilde E^{-1} E {\mathbf U}_0 , E{\mathbf U}_0)}{(E {\mathbf U}_0,{\mathbf U}_0)}| \,.
\end{array}
\]
\end{proof}
This means that formula~\eqref{eq:crgeneoACS} for $c_R$ can be expressed explicitely in terms of $k_0$ and of the minimal and maximal eigenvalue of $\tilde E^{-1} E$.\\

$\bullet$ Stable decomposition\\
Let $\mathbf{U}\in \HO$ be decomposed as follows:
\[
\begin{array}{lcl}
	\mathbf{U} &=& P_0 \mathbf{U} + (\id-P_0) \mathbf{U} = P_0 \mathbf{U} + (\id-P_0) \sum_{j=1}^N R_j^T D_j R_j \mathbf{U} \\
	&=& P_0 \mathbf{U} + (\id-P_0) \sum_{j=1}^N R_j^T D_j (\id-\tilde \pi_j) R_j \mathbf{U} + \underbrace{(\id-P_0) \sum_{j=1}^N R_j^T D_j \tilde \pi_j\, R_j \mathbf{U}}_{=0}\\	
    &=& \underbrace{P_0 \mathbf{U} + (\tilde P_0-P_0) \sum_{j=1}^N R_j^T D_j (\id-\tilde \pi_j) R_j \mathbf{U}}_{:= F\mathbf{U}\,\in  V_0}
 +  (\id-\tilde P_0) \sum_{j=1}^N R_j^T D_j (\id-\tilde \pi_j) R_j \mathbf{U}\,.\\
\end{array}
\]
Let ${\mathbf U}_0\in\R^{\#{\mathcal N}_0}$ be such that $Z {\mathbf U}_0 = F\mathbf{U}$, we choose the following decomposition:
\[
\mathbf{U} = \widetilde \RL({\mathbf U}_0 , (D_j (\id-\tilde \pi_j) R_j \mathbf{U})_{1\le j \le N} )\,.
\]
The stable decomposition consists in estimating a constant $c_T>0$ such that: 
\begin{equation}
	\label{eq:stabledec}
	c_T\,[(\tilde E\,{\mathbf U}_0,{\mathbf U}_0) + \sum_{j=1}^N (R_j A R_j^T D_j (\id-\tilde \pi_j) R_j \mathbf{U} , D_j (\id-\tilde \pi_j) R_j \mathbf{U})]
\leq a( \mathbf{U} , \mathbf{U})\,.
\end{equation}
Since the second term in the left hand side is the same as in the exact coarse solve method, we have (see~\cite{Dolean:2015:IDDSiam}, page~177, Lemma~7.15):
\begin{equation}
	\label{eq:kuninequality}
	\sum_{j=1}^N (R_j A R_j^T D_j (\id-\tilde \pi_j) R_j \mathbf{U} , D_j (\id-\tilde \pi_j) R_j \mathbf{U}) \leq k_1\,\tau\, a( \mathbf{U} , \mathbf{U} )\,.
\end{equation}
We now focus on the first term of the left hand side of \eqref{eq:stabledec}. 
Let $\delta$ be some positive number, using again~\eqref{eq:kuninequality}, the following auxiliary result holds:
\[
\begin{array}{lcl}
\|F\mathbf{U}\|_A^2 &\leq& (1+\delta) \|P_0\mathbf{U}, P_0\mathbf{U}\|_A^2\\[.4em]
& &+ (1+\frac{1}{\delta}) \| (P_0-\tilde P_0) \sum_{j=1}^N R_j^T D_j (\id-\tilde \pi_j) R_j \mathbf{U} \|_A^2\\[.9em]
&\leq& (1+\delta) (A\mathbf{U},\mathbf{U})\\[.4em]
& & + (1+\frac{1}{\delta}) \| (P_0-\tilde P_0)\|_A^2\,\, 
\| \sum_{j=1}^N R_j^T D_j (\id-\tilde \pi_j) R_j \mathbf{U} \|_A^2\\[.9em]
&\leq& (1+\delta) (A\mathbf{U},\mathbf{U})\\[.4em]
& & + (1+\frac{1}{\delta}) \| (P_0-\tilde P_0)\|_A^2\,\, 
k_0 \sum_{j=1}^N \| R_j^T D_j (\id-\tilde \pi_j) R_j \mathbf{U} \|_A^2\\[.9em]
&\leq& \left(1+\delta + (1+\frac{1}{\delta}) \| (P_0-\tilde P_0)\|_A^2 k_0 k_1\,\tau\right)\, a(\mathbf{U},\mathbf{U})
\end{array}
\]
The best possible value for $\delta$ is 
\[
\delta := \epsilon_A \sqrt{k_0\,k_1\,\tau}\,.
\]
Hence, we have:
\begin{equation}
	\label{eq:majorationFU}
	 (Z^T\,A\,Z \mathbf{U}_0\,,\,\mathbf{U}_0) = \|F\mathbf{U}\|_A^2 \leq (1+\epsilon_A \sqrt{k_0\,k_1\,\tau})^2   a(\mathbf{U},\mathbf{U})\,.
\end{equation}
Thus, we have: 
\[
\begin{array}{rcl}
(\tilde E\, \mathbf{U}_0\,,\,\mathbf{U}_0) &=& (\tilde E\, E^{-1/2}E^{1/2} \mathbf{U}_0\,,\,E^{-1/2}E^{1/2}\mathbf{U}_0) = (E^{-1/2} \tilde E\, E^{-1/2}E^{1/2} \mathbf{U}_0\,,\,E^{1/2}\mathbf{U}_0) \\
&\le& \lambda_{max}(E^{-1/2} \tilde E\, E^{-1/2})\,(E^{1/2} \mathbf{U}_0\,,\,E^{1/2}\mathbf{U}_0) \\	
&=& \lambda_{max}(E^{-1} \tilde E) (Z^T\,A\,Z \mathbf{U}_0\,,\,\mathbf{U}_0)\\
&\le&  \lambda_{max}(E^{-1} \tilde E) (1+\epsilon_A \sqrt{k_0\,k_1\,\tau})^2   a(\mathbf{U},\mathbf{U})\,.
\end{array}
\]
This last estimate along with \eqref{eq:kuninequality}  prove that in \eqref{eq:stabledec}, it is possible to take 
\begin{equation}
	\label{eq:ctgeneoACS}
\displaystyle  c_T = \frac{\lambda_{min}(E \tilde E^{-1})}{ (1+\epsilon_A \sqrt{k_0\,k_1\,\tau})^2 + \lambda_{min}(E \tilde E^{-1}) k_1\tau }\,.	
\end{equation}
Overall, with $c_T$ given by \eqref{eq:ctgeneoACS} and $c_R$ by \eqref{eq:crgeneoACS}, we have proved the following spectral estimate:
\begin{equation}
	\label{eq:specgeneoACS}
	c_T \le \lambda(M^{-1}_{GenEOACS}\,A) \le  c_R\,.		
\end{equation}
Constants $c_T$ and $c_R$ are stable with respect to $\epsilon_A$ and the spectrum of $E \tilde E^{-1}$ so that \eqref{eq:specgeneoACS} proves the stability of preconditioner $M^{-1}_{GenEOACS}$ w.r.t. inexact solves. 

\section{Inexact Coarse Solves for GenEO2} 
\label{sec:geneo2_with_inexact_coarse_space}

The GenEO-2 coarse space construction was introduced in~\cite{Haferssas:2015:RCS,haferssas:2016:ASM} , see~\cite{Dolean:2015:IDDSiam} also \S~7.7, page~186. It is motivated by domain decomposition methods for which the local solves are not necessarily Dirichlet solves e.g. discretization of Robin boundary value problems, see~\cite{STCYR:2007:OMA}. We have not been able to prove the robustness of the  GenEO-2 coarse space with respect to inexact coarse solves when used in the original GenEO-2 preconditioner~\eqref{eq:Mm1geneo2notrobust}, see remark~\ref{rk:geneo2notrobust}. For this reason, we study here a slight modification of the preconditioner, eq.~\eqref{eq:Mm1geneo2}, for which we prove robustness. { The more intricate analysis of GenEO2 compared to the one of GenEO is related to the differences between the Schwarz and P.L.~Lions algorithms themselves. Indeed, in the Schwarz method, Assumption~(ii) of the fictitious space lemma~\ref{th:fictitiousSpaceLemma} comes almost for free even for a one level method whereas Assumption~(iii) (stable decomposition) can only be fulfilled with a two level method. In P.L.~Lions algorithm neither of the two assumptions are satisfied by the one level method. This is reflected in the fact that the proofs for GenEO2 are more intricate than for GenEO.}

For all subdomains $1\le i \le N$, let $B_i$ be a matrix of size ${\#\mathcal N_i}\times {\#\mathcal N_i}$, which comes typically from the discretization of boundary value local problems using optimized transmission conditions or Neumann boundary conditions. Recall that by construction matrix $D_i\,R_i A R_i^T D_i$ is symmetric positive-semi definite and we make the extra following assumption: 
\begin{assumption}
	\label{as:BiSP}
  For all subdomains $1\le i\le N$, matrix $B_i$ is symmetric positive semi-definite and either of the two conditions holds
  \begin{itemize}
  	\item $B_i$ is definite,
	\item $B_i=\widetilde A^i$  and $D_i\,R_i A R_i^T D_i$ is definite.
  \end{itemize}	
\end{assumption} 
We first consider the case where $B_i$ is definite. The other case will be treated in Remark~\ref{rq:NeuNeu}. 
We recall the coarse space defined in~\cite{Haferssas:2015:RCS,haferssas:2016:ASM,Dolean:2015:IDDSiam}. Let $\gamma$ and $\tau$ be two user defined thresholds. We introduce two generalized eigenvalue problems which by Assumption~\ref{as:BiSP} are regular.

\begin{definition}[Generalized Eigenvalue Problem for the lower bound]
	\label{th:tauthreshold}
For each subdomain $1\le j\le N$, 
we introduce the generalized eigenvalue problem
\begin{equation}
  \label{eq:eigAtildeB}
	{\begin{array}{c}
	\mbox{ Find }(\mathbf{V}_{jk},\lambda_{jk})\in
	\R^{\# {\mathcal N}_j}\setminus \{0\} \times \R
	\mbox{ such that}\\
	\widetilde A^j  \mathbf{V}_{jk} = \lambda_{jk} B_j \mathbf{V}_{jk} 
	\, \,.
	\end{array}}
\end{equation}
Let $\tau>0$ be a user-defined threshold and $\tilde\pi_j$ be the projection from $\R^{\# {\mathcal N}_j}$ on $V_{j\tau}=\text{Span}\{ \mathbf{V}_{jk} | \lambda_{jk} < \tau \}$ parallel to $\text{Span}\{ \mathbf{V}_{jk} | \lambda_{jk} \ge \tau \}$. We define $V_{j,geneo}^\tau\subset\R^{\# {\mathcal N}}$ as the vector space spanned by the family of vectors $(R_j^T D_j {\mathbf V}_{jk})_{\lambda_{jk} < \tau}$ corresponding to eigenvalues smaller than $\tau$. 
Let $V_{geneo}^\tau$ be the vector space spanned by the collection over all subdomains of vector spaces $(V_{j,geneo}^\tau)_{1\le j \le N}$. 
\end{definition}

\bigskip 

\begin{definition}[Generalized Eigenvalue Problem for the upper bound]
For each subdomain $1\le i\le N$, 
we introduce the generalized eigenvalue problem
	\label{def:gammathreshold}
	\begin{equation}
		\label{eq:thCG:eig_geneoBNNnew}
	\begin{array}{c}
	\text{Find } ({\mathbf U}_{ik},\mu_{ik})\in \R^{\# {\mathcal N}_i}\setminus \{0\} \times \R \mbox{ such that}\\[2ex]
	D_i R_i A R_i^T D_i {\mathbf U}_{ik} = \mu_{ik} B_i\, {\mathbf U}_{ik} \, 
	\,\,.
	\end{array}
	\end{equation}
	Let $\gamma>0$ be a user-defined threshold, we define $V_{i,geneo}^\gamma\subset\R^{\# {\mathcal N}}$ as the vector space spanned by the family of vectors $(R_i^T D_i {\mathbf U}_{ik})_{\mu_{ik} > \gamma}$ corresponding to eigenvalues larger than $\gamma$. Let $V_{geneo}^\gamma$ be the vector space spanned by the collection over all subdomains of vector spaces $(V_{j,geneo}^\gamma)_{1\le j \le N}$.
\end{definition}
Now, let $\xi_i$ denote the $B_i$-orthogonal projection from $\R^{\#{\mathcal N}_i}$ on 
\[
   V_{i\,\gamma}:=\text{Span} \left\{{\mathbf U}_{ik}\, |\, \gamma < \mu_{ik}  \right\}
\]
 parallel to 
 \[
     W_{i\,\gamma}:=\text{Span} \left\{{\mathbf U}_{ik}\, |\, \gamma \ge \mu_{ik}  \right\}\,.
 \]
The coarse space $V_0$ built from the above generalized eigenvalues is defined as the following sum:
\[
V_0 := V_{geneo}^\tau + V_{geneo}^\gamma\,.
\]
It is spanned by the columns of a full rank rectangular matrix $Z=R_0^T$ with $\#{\mathcal N}_0$ columns. Projection $P_0$ and its approximation $\tilde P_0$ are defined by the same formula as above, see~\eqref{eq:potilde}. 

We have the following
\begin{lemma}
	\label{th:pjproj}
For $1\le j \le N$, let us introduce the $B_j$-orthogonal projection $p_j$ from $\R^{\#{\mathcal N}_j}$ on 
\[
  V_{j,\tau\gamma} := V_{j,\tau} \oplus V_{j,\gamma}\,.
\] 
Then for all ${\mathbf U}_j\in  \R^{\#{\mathcal N}_j}$, we have:
	\[
	\tau \, (B_j (\id-p_j) {\mathbf U}_j , (\id-p_j) {\mathbf U}_j )     \leq (\tilde A_j {\mathbf U}_j , {\mathbf U}_j)\,.
	\]
Moreover, for all $\mathbf{U}\in  \R^{\#{\mathcal N}}$, we have:
\[
 \tau\,\sum_{j=1}^N(B_j (\id-p_j) R_j \mathbf{U} , (\id-p_j) R_j \mathbf{U} )     \leq k_1\,a(\mathbf{U} , \mathbf{U})\,.
\]
\end{lemma}
\begin{proof}
	Let ${\mathbf U}_j\in \R^{\#{\mathcal N}_j}$, we have:
\[
\begin{array}{rcl}
	(B_j (\id-\tilde\pi_j) {\mathbf U}_j , (\id-\tilde\pi_j) {\mathbf U}_j ) 
	&=&  (B_j (\id- p_j +(p_j -\tilde\pi_j)) {\mathbf U}_j , (\id- p_j +(p_j -\tilde\pi_j)) {\mathbf U}_j ) \\
&=& \| (\id- p_j) {\mathbf U}_j \|_{B_j}^2 + \| (p_j-\tilde \pi_j) {\mathbf U}_j \|_{B_j}^2 \\
& &+ 2  \underbrace{(B_j (\id- p_j) {\mathbf U}_j , (p_j -\tilde\pi_j) {\mathbf U}_j )}_{= 0 \text{ since $\tilde \pi_j {\mathbf U}_j\in V_{j\,,\tau}$}\subset V_{j,\tau\gamma}}\\
&\ge& \| (\id- p_j) {\mathbf U}_j \|_{B_j}^2 = (B_j (\id-p_j) {\mathbf U}_j , (\id-p_j) {\mathbf U}_j ) \,.	
\end{array}
\]	
Since we have by Lemma~7.6, page~167 in \lecnot:
\[
 \tau \, (B_j (\id-\tilde\pi_j) {\mathbf U}_j , (\id-\tilde\pi_j) {\mathbf U}_j )  \leq (\tilde A_j {\mathbf U}_j , {\mathbf U}_j)\,,
\]
the conclusion follows by summation over all subdomains. 
\end{proof}

\newcommand{\bw}{b_{W_i}}
\newcommand{\Bw}{B_{W_i}}

The definition of the stable preconditioner is based on a pseudo inverse of $B_i$ that we introduce now. Let $\bw$ denote the restriction of $b_i$ to  $W_{i\,\gamma}\times W_{i\,\gamma}$ where $W_{i\,\gamma}$ is endowed with the Euclidean scalar product:
	\label{eq:bw}
	\begin{align}
	\bw :& W_{i\,\gamma}\times W_{i\,\gamma} \longrightarrow \R\nonumber\\
		 & ({\mathbf U}_i\,,\,{\mathbf V}_i) \mapsto b_i({\mathbf U}_i\,,\,{\mathbf V}_i)\,.
	\end{align}
By Riesz representation theorem, there exists a unique isomorphism $\Bw: W_{i\,\gamma} \longrightarrow W_{i\,\gamma}$ into itself so that for all ${\mathbf U}_i\,,\,{\mathbf V}_i\in W_{i\,\gamma}$, we have:
\[
  \bw({\mathbf U}_i\,,\,{\mathbf V}_i) = (\Bw\,{\mathbf U}_i\,,\,{\mathbf V}_i)
  \,.
\]
The inverse of $\Bw$ will be denoted by $B_i^\dag$ and is given by the following formula
\begin{equation}
	B_i^\dag = (\id - \xi_i)B_i^{-1}\,.
\end{equation}
In order to check this formula, we have to show that $ \Bw(\id - \xi_i)B_i^{-1} y= y$ for all  $y\in W_{i\,\gamma}$. Let $z\in W_{i\,\gamma}$, using the fact that $\id-\xi_i$ is the $b_i$-orthogonal projection on $W_{i\,\gamma}$, we have:
\begin{align}
	( \Bw(\id - \xi_i)B_i^{-1} y , z) = b_i( (\id - \xi_i)B_i^{-1} y , z) = b_i( B_i^{-1} y , z) = (y,z)\,.
\end{align}
Since this equality holds for any $z\in W_{i\,\gamma}$, this proves that $ \Bw(\id - \xi_i)B_i^{-1} y= y$.\\[.5em]

We study now the preconditioner given by:
\begin{definition}[Preconditioner $M^{-1}_{GenEO2ACS}$]
	\label{def:precGenEO2ACS}
Let $q_i$ denote the orthogonal projection {(for the Euclidean scalar product)} from $\R^{\#{\mathcal N}_i}$ onto $W_{i\,\gamma}$. We define the preconditioner $M^{-1}_{GenEO2ACS}$ as follows:
\begin{align}
	\label{eq:Mm1geneo2}
  M^{-1}_{GenEO2ACS} &:=  Z\, \tilde E^{-1}\,Z^T \nonumber\\
  &+   (\id - \tilde P_0)\, \left(\sum_{i=1}^N R_i^T\,D_i\,q_i\,B_i^\dag\,q_i\, D_i\,\,R_i\right)\, (\id - \tilde P_0^T)\,.  
\end{align}	
\end{definition}

\begin{remark}
In order to write an explicit form for the projection $q_i$, we denote by $Z_{i\,\gamma}$ a rectangular matrix whose columns are a basis for $V_{i\,\gamma}$. Let ${\mathbf U}_i\in\R^{\#{\mathcal N}_i}$  be a vector we want to project. The projection $q_i {\mathbf U}_i$ is the solution to the constrained minimization problem:
\[
  \min_{{\mathbf W}_i \in W_{i\,\gamma}} \frac{1}{2} \| {\mathbf W}_i - {\mathbf U}_i  \|^2 = \min_{{\mathbf W}_i | (B_i Z_{i\,\gamma})^T{\mathbf W}_i = 0 } \frac{1}{2} \| {\mathbf W}_i - {\mathbf U}_i  \|^2
\]
Using the Lagrange multiplier technique, we introduce $\lambda\in\R^{dim(V_{i\,\gamma})}$ and the optimality conditions read:
\[
\begin{array}{rcl}
(q_i{\mathbf U}_i  - {\mathbf U}_i)	+ B_i Z_{i\,\gamma}\lambda &=& 0\\
	(B_i Z_{i\,\gamma})^T q_i{\mathbf U}_i  &=& 0\,.
\end{array}
\]
The vector $\lambda$ must satisfy
\[
   -  (B_i Z_{i\,\gamma})^T  {\mathbf U}_i	+ (B_i Z_{i\,\gamma})^T  B_i Z_{i\,\gamma}\lambda = 0\,.
\]

Finally, an explicit formula for the projection $q_i$ is:
	\[
q_i = \id - B_i Z_{i\,\gamma} ( (B_i Z_{i\,\gamma})^T B_i Z_{i\,\gamma} )^{-1} Z_{i\,\gamma}^T B_i^T \,.	
	\]
{Thus applying $q_i$ amounts to solving concurrently in each subdomain a small linear system of size the number of local eigenvectors contributing to the coarse space.}
\end{remark}

\begin{remark}
Note that  $q_i \,B_i^\dag$ is actually equal to $B_i^\dag$ but its presence shows the symmetry of the preconditioner. 
\end{remark}

We can now define the abstract framework for the preconditioner. Let $\HP$ be defined by
\[
	\HP := \R^{\#{\mathcal N}_0} \times \Pi_{i=1}^N W_{i\,\gamma}
\]
 endowed with the following bilinear form arising from local SPD matrices $(B_i)_{1\le i \le N}$
\begin{equation}
  \label{eq:bHPapprox2}
  \begin{array}{rcl}
  \tilde b: \HP \times \HP &\longrightarrow &\R \\    
\phantom{b: }({\mathcal U},{\mathcal V}) &\longmapsto& b({\mathcal
  U},{\mathcal V}):= (\tilde E {\mathbf U}_0,\,{\mathbf V}_0) + \sum_{i=1}^N (B_i\,{\mathbf U}_i,\,{\mathbf V}_i)\,
  \end{array}
\end{equation}
We denote by $\tilde B: \HP \rightarrow \HP$ the operator such that $(\tilde Bu_D,v_D)_D=\tilde b(u_D,v_D)$ for all $u_D,v_D\in\HP$.\\ 
Let $\widetilde\RL: \HP \longrightarrow \HO$ be defined using operator $\tilde P_0$ (see eq.~\eqref{eq:potilde}):
\begin{equation}
  \label{eq:RL3} 
\widetilde\RL({\cal U}):= Z\, \mathbf{U}_0 +  (\id- \tilde P_0 )\,\sum_{i=1}^N R_i^T D_i {\mathbf U}_i\,.
\end{equation}
Recall that if we had used an exact coarse space solve, we would have introduced:
\begin{equation}
  \label{eq:RLexact3} 
\RL({\cal U}):= Z\, {\mathbf U}_0 +  (\id- P_0 )\,\sum_{i=1}^N R_i^T D_i {\mathbf U}_i\,.
\end{equation}
Note that we have 
\[
	  \widetilde\RL({\cal U}) = \RL({\cal U}) + (P_0-\tilde P_0) \,\sum_{i=1}^N R_i^T D_i U_i\,.
\]
It can be checked that the resulting preconditioner with inexact coarse solve $M^{-1}_{GenEO2ACS}$ (Eq.~\eqref{eq:Mm1geneo2}) satisfies $M^{-1}_{GenEO2ACS}= \widetilde\RL\, \widetilde B^{-1} \,\widetilde\RL^T$. Indeed, we have:
\[
  \widetilde\RL^T {\mathbf V} = (Z^T {\mathbf V} ,  (q_i D_i R_i (\id - \tilde P_0^T){\mathbf V})_{1\le i\le N} ) 
\]

\paragraph{Auxiliary results on GEVP} 
\label{par:auxiliary_results_on_gevp}

Beware, in this paragraph, $A$ and $B$ have nothing to do with the global problem to be solved:
\begin{lemma}
	\label{th:esteig}
	Let $A$ be a symmetric positive semi definite matrix and $B$ be a symmetric positive definite matrix. We consider the generalized eigenvalue problem:
\[
   A {\mathbf U} = \lambda B {\mathbf U}\,.
\]
The generalized eigenvectors and eigenvalues are denoted by $({\mathbf  U}_k,\lambda_k)_{k\ge 1}$. Let $\tau$ be a positive number. We define
\[
	V_{\tau}:=\text{Span}\{ {\mathbf U}_k\,| \lambda_k < \tau\,\}\,.
\]
Let $W$ be any linear subspace. We denote by $p$ the $B$-orthogonal projection on $V_{\tau}\bigplus W$.\\
Then, for all ${\mathbf U}$ we have the following estimate:
\begin{equation}
	\label{eq:esteig}
	\tau\, (B\,(\id-p) {\mathbf U}\,,\, (\id-p) {\mathbf U})\ \le (A (\id-p) {\mathbf U} , (\id-p) {\mathbf U})\,.
\end{equation}
Similarly, let $\gamma$ be a positive number. We define 
\[
	V_{\gamma}:=\text{Span}\{ {\mathbf U}_k\,|\, \lambda_k > \gamma\,\}\,.
\]
Let $W$ be any linear subspace. We denote by $q$ the $B$-orthogonal projection on $V_{\gamma}\bigplus W$.\\
Then, for all ${\mathbf U}$ we have the following estimate:
\begin{equation}
	\label{eq:esteigup}
(A (\id-q) {\mathbf U} , (\id-q) {\mathbf U}) \le \gamma\,(B\,(\id-q) {\mathbf U}\,,\, (\id-q) {\mathbf U})\,.
\end{equation}
\end{lemma}
\begin{proof}
 We have using $V_{\tau}\subset V_{\tau}\bigplus W$:
	\begin{equation*}
		\tau \le \min_{{\mathbf U} \in V_{\tau}^{B\bot} } \frac{(A{\mathbf U},{\mathbf U})}{(B{\mathbf U} , {\mathbf U})} \le \min_{{\mathbf U}  \in (V_{\tau}\bigplus W)^{B\bot}} \frac{(A{\mathbf U} , {\mathbf U})}{(B{\mathbf U} , {\mathbf U})}\,.
	\end{equation*}
For all ${\mathbf U}$, the vector $(\id-p) {\mathbf U}$ is $B$-orthogonal to $V_{\tau}\bigplus W$ and this ends the proof of~\eqref{eq:esteig}. The proof of~\eqref{eq:esteigup} follows similarily from 
\begin{equation*}
	\gamma \ge \max_{{\mathbf U} \in V_{\gamma}^{B\bot}} \frac{(A{\mathbf U},{\mathbf U})}{(B{\mathbf U} , {\mathbf U})}\,.
\end{equation*}
\end{proof}

\bigskip


In order to apply the fictitious space Lemma~\ref{th:fictitiousSpaceLemma} to the study of the preconditioner \eqref{eq:Mm1geneo2}, three assumptions have to be checked. \\

$\bullet$ $\widetilde\RL$ is onto.\\
Let $\mathbf{U}\in \HO$, we have  
\[
\begin{array}{rcl}
  \mathbf{U} &=& \tilde P_0 \mathbf{U} + (\id- \tilde P_0  )\,\mathbf{U} \\
&=& \tilde P_0 \mathbf{U} + (\id- \tilde P_0  )\,\sum_{i=1}^N R_i^T D_i R_i \mathbf{U} \\
&=& \tilde P_0 \mathbf{U} + (\id- \tilde P_0  )\,\sum_{i=1}^N R_i^T D_i \xi_i R_i \mathbf{U} + (\id- \tilde P_0  )\,\sum_{i=1}^N R_i^T D_i (\id - \xi_i) R_i \mathbf{U}\\
&=& \underbrace{\tilde P_0 \mathbf{U} + (P_0 - \tilde P_0) \sum_{i=1}^N R_i^T D_i \xi_i R_i \mathbf{U}}_{:=F\mathbf{U}} + \underbrace{(\id- P_0  )\,\sum_{i=1}^N R_i^T D_i \xi_i R_i \mathbf{U}}_{=0}\\[2.7em]
& & + (\id- \tilde P_0  )\,\sum_{i=1}^N R_i^T D_i (\id - \xi_i) R_i \mathbf{U}\,.
\end{array}
\]
Let us consider the last equality. Since $F\mathbf{U}$ is the sum of two terms that belong to ${\mathbf V}_0$ there exists ${\mathbf U}_0$ such that $Z {\mathbf U}_0 = F\mathbf{U}$. The third term is zero since $\sum_{i=1}^N R_i^T D_i \xi_i R_i \mathbf{U}\in V_0$. Note also that $(\id - \xi_i) R_i \mathbf{U}\in W_{i\,\gamma}$. Therefore, we have 
\[
   \mathbf{U} = \widetilde\RL({\mathbf U}_0, ((\id - \xi_i)\,R_i \mathbf{U})_{1\leq i\leq N} ))\,.
\]	

$\bullet$ Continuity of $\widetilde\RL$\\

We have to estimate a constant $c_R$ such that for all ${\mathcal U}=({\mathbf U}_0\,,\,({\mathbf U}_i)_{1\le i \le N})\in\HP$ we have:
\[
\begin{array}{rcl}
a(\widetilde\RL ({\mathcal U}),\widetilde\RL ({\mathcal U})) &\le& c_R\, \tilde b({\mathcal U}\,,\,{\mathcal U})\,\\[.9em]
&=& c_R [ (\tilde E\mathbf{U}_0 , \mathbf{U}_0) + \sum_{i=1}^N ( B_i \mathbf{U}_i , \mathbf{U}_i)]\,.
\end{array}
\]
Note that using $(\id - \xi_i) {\mathbf U}_i = {\mathbf U}_i$ (recall that ${\mathbf U}_i\in W_{i\gamma}$), we have:
\[
	\begin{array}{rcl}
	\tilde \RL ({\mathcal U}) &=& Z {\mathbf U}_0 + (\id-\tilde P_0) \sum_{i=1}^N R_i^T\,D_i\,{\mathbf U}_i\\[.9em]
	&=& Z {\mathbf U}_0 + (P_0-\tilde P_0)\sum_{i=1}^N R_i^T\,D_i\,{\mathbf U}_i + (\id- P_0) \sum_{i=1}^N R_i^T\,D_i\,{\mathbf U}_i \\[.9em]
	&=& \underbrace{Z {\mathbf U}_0 + (P_0-\tilde P_0)\sum_{i=1}^N R_i^T\,D_i\,(\id-\xi_i){\mathbf U}_i}_{\in V_0} + (\id- P_0) \sum_{i=1}^N R_i^T\,D_i\,(\id-\xi_i) {\mathbf U}_i
\end{array}
\]
We have thus the following estimate using the $A$-orthogonality of $\id-P_0$:
\[
	\begin{array}{rcl}
	a(\widetilde\RL ({\mathcal U}),\widetilde\RL ({\mathcal U})) 
	&=& \| Z {\mathbf U}_0 + (P_0-\tilde P_0)\sum_{i=1}^N R_i^T\,D_i\, (\id-\xi_i) {\mathbf U}_i \\
	& & + (\id-P_0) \sum_{i=1}^N R_i^T\,D_i\,(\id-\xi_i) {\mathbf U}_i \|_A^2\\[.9em]
	&=&  \| Z {\mathbf U}_0 + (P_0-\tilde P_0)\sum_{i=1}^N R_i^T\,D_i\, (\id-\xi_i) {\mathbf U}_i \|_A^2 \\
	& & + \| (\id-P_0) \sum_{i=1}^N R_i^T\,D_i\,(\id-\xi_i) {\mathbf U}_i \|_A^2\\[.9em]
	&\le& (1+\delta) \| Z {\mathbf U}_0 \|_A^2 + (1+\frac{1}{\delta})\| (P_0-\tilde P_0)\sum_{i=1}^N R_i^T\,D_i\,(\id-\xi_i) {\mathbf U}_i \|_A^2 \\
	& & + \| \sum_{i=1}^N R_i^T\,D_i\,(\id-\xi_i) {\mathbf U}_i \|_A^2\\[.9em]
	&\le& (1+\delta) (E\, {\mathbf U}_0\,,\,{\mathbf U}_0) + k_0 \sum_{i=1}^N \| R_i^T\,D_i\,(\id-\xi_i) {\mathbf U}_i \|_A^2\\ 
	& & + (1+\frac{1}{\delta})\| (P_0-\tilde P_0)\|_A^2 k_0 \sum_{i=1}^N \| R_i^T\,D_i\, (\id-\xi_i) {\mathbf U}_i \|_A^2 \\[.9em]
	&\le& (1+\delta) \lambda_{max}(E \tilde E^{-1}) (\tilde E\, {\mathbf U}_0\,,\,{\mathbf U}_0) \\
	& & + k_0\,\gamma (1 + (1+\frac{1}{\delta})\| (P_0-\tilde P_0)\|_A^2) \sum_{i=1}^N \,(B_i\,(\id-\xi_i) {\mathbf U}_i\,,\,(\id-\xi_i) {\mathbf U}_i ) \\[.9em]
	&\le& \max( (1+\delta)\,\lambda_{max}(E \tilde E^{-1}) \,,\, k_0\,\gamma\, (1 + (1+\frac{1}{\delta})\,\epsilon_A^2  )\,\, \tilde b( {\mathcal U} , {\mathcal U} )\,.
	\end{array}
\]
Based on Lemma~\ref{th:optannexe}, we can optimize the value of $\delta$ and take 
\begin{equation}
	\label{eq:crgeneo2ACS}
	c_R := \frac{  k_0\,\gamma\, (1 + \epsilon_A^2) + \lambda_{max}(E \tilde E^{-1}) + \sqrt{ (k_0\,\gamma\, (1 + \epsilon_A^2) - \lambda_{max}(E \tilde E^{-1}))^2 + 4 \lambda_{max}(E \tilde E^{-1}) k_0\gamma \epsilon_A^2 }   }{2}\,.
\end{equation}

$\bullet$ Stable decomposition

The stable decomposition estimate is based on using projections $p_j$ defined in Lemma~\ref{th:pjproj}. Let $\mathbf{U}\in \HO$ be decomposed as follows:
\[
\begin{array}{lcl}
	\mathbf{U} &=& P_0 \mathbf{U} + (\id-P_0) \sum_{j=1}^N R_j^T D_j (\id-p_j) R_j \mathbf{U} + \underbrace{(\id-P_0) \sum_{j=1}^N R_j^T D_j p_j R_j \mathbf{U}}_{=0}\\
 &=& \underbrace{P_0 \mathbf{U} + (\tilde P_0-P_0) \sum_{j=1}^N R_j^T D_j (\id-p_j) R_j \mathbf{U}}_{:= F\mathbf{U}\,\in  V_0}
 +  (\id-\tilde P_0) \sum_{j=1}^N R_j^T D_j (\id-p_j) R_j \mathbf{U}\,.\\
\end{array}
\]
We define ${\mathbf U}_0$ be such that $Z {\mathbf U}_0 = F\mathbf{U}$. We have that $(\id-p_j)R_j \mathbf{U}$ is $B_j$-orthogonal to $V_{\gamma\,j}\bigplus V_{\tau\,j}$ and thus to $V_{\gamma\,j}$. This means that $(\id-p_j)R_j \mathbf{U}\in W_{\gamma\,j}$ and that we can choose the following decomposition:
\[
\mathbf{U} = \widetilde \RL({\mathbf U}_0 , ( (\id-p_j) R_j \mathbf{U})_{1\le j \le N} )\,.
\]
The stability of the decomposition consists in estimating a constant $c_T>0$ such that :
\begin{equation}
	\label{eq:stabledec3}
	c_T\,[(\tilde E\,{\mathbf U}_0,{\mathbf U}_0) + \sum_{j=1}^N (B_j (\id-p_j) R_j \mathbf{U} ,  (\id-p_j) R_j \mathbf{U})]
\leq a( \mathbf{U} , \mathbf{U})\,.
\end{equation}
Using Lemma~\ref{th:pjproj}, we have
\begin{equation}
	\label{eq:esttaukunpj0}
	\tau\, \sum_{j=1}^N (B_j (\id- p_j) R_j \mathbf{U} , (\id- p_j) R_j \mathbf{U}) \leq k_1\,a( \mathbf{U} , \mathbf{U} )\,.
\end{equation}
We now focus on the first term of the left hand side of \eqref{eq:stabledec3}. 
Let $\delta$ be some positive number, the following auxiliary result will be useful:
\[
\begin{array}{lcl}
\|F\mathbf{U}\|_A^2 &\leq& (1+\delta) \|P_0\mathbf{U}, P_0\mathbf{U}\|_A^2\\[.4em]
& &+ (1+\frac{1}{\delta}) \| (P_0-\tilde P_0) \sum_{j=1}^N R_j^T D_j (\id-p_j) R_j \mathbf{U} \|_A^2\\[.9em]
&\leq& (1+\delta) (A\mathbf{U},\mathbf{U})\\[.4em]
& & + (1+\frac{1}{\delta}) \| (P_0-\tilde P_0)\|_A^2 
\| \sum_{j=1}^N R_j^T D_j (\id-p_j) R_j \mathbf{U} \|_A^2\\[.9em]
&\leq& (1+\delta) a(\mathbf{U},\mathbf{U}) \\[.4em]
& & + (1+\frac{1}{\delta}) \| (P_0-\tilde P_0)\|_A^2\, k_0 \sum_{j=1}^N \| R_j^T D_j (\id-p_j) R_j \mathbf{U} \|_A^2\\[.9em]
&\leq& (1+\delta) a(\mathbf{U},\mathbf{U}) \\[.4em]
& & + (1+\frac{1}{\delta}) \| (P_0-\tilde P_0)\|_A^2\, k_0 \gamma \sum_{j=1}^N  (B_j (\id-p_j)\, R_j\,\mathbf{U} , (\id-p_j)\, R_j\,\mathbf{U})\,\\[.9em]
&\leq& ((1+\delta) + (1+\frac{1}{\delta}) \| (P_0-\tilde P_0)\|_A^2\, k_0 \gamma\,\tau^{-1} k_1)\, a(\mathbf{U},\mathbf{U})\,
\end{array}
\]
where we have used Lemma~\ref{th:esteig} \eqref{eq:esteigup} (applied with $A$ replaced by $D_j\,R_j\, A\, R_j^T\, D_j$ and $B$ by $B_j$) for the one before last estimate and Lemma~\ref{th:pjproj} for the last estimate.\\

The optimal value for $\delta$ yields:
\begin{equation}
	\label{eq:FUestgeneo2}
	\|F\mathbf{U}\|_A^2 \le (1+\epsilon_A\,\sqrt{k_0\,k_1 \gamma\,\tau^{-1} })^2  a(\mathbf{U},\mathbf{U})\,.
\end{equation}
We have 
\[
\begin{array}{lcl}
   (\tilde E{\mathbf U}_0 , {\mathbf U}_0) &\leq& \lambda_{max} (E^{-1} \tilde E) (E{\mathbf U}_0 , {\mathbf U}_0) = \lambda_{max} (E^{-1} \tilde E) (A\,Z{\mathbf U}_0 , Z{\mathbf U}_0)\\
    &=& \lambda_{max} (E^{-1} \tilde E)  \|F\mathbf{U}\|_A^2\,.
\end{array}   
\]
so that with \eqref{eq:FUestgeneo2}, this yields:
\[
(\tilde E{\mathbf U}_0\, ,\,{\mathbf U}_0) \leq \lambda_{max} (E^{-1} \tilde E)\,(1+\epsilon_A \sqrt{k_0\,k_1\,\gamma\,\tau^{-1}})^2  \, a(\mathbf{U},\mathbf{U})\,.
\]
Finally, in \eqref{eq:stabledec3} we can take :
\begin{equation}
	\label{eq:ctestimategeneo2}
	c_T := \frac{1}{\lambda_{max} (E^{-1} \tilde E)\,(1+\epsilon_A \sqrt{k_0\,k_1\,\gamma\,\tau^{-1}})^2 + k_1\,\tau^{-1}}\,.
\end{equation}

Overall, with $c_T$ given by \eqref{eq:ctestimategeneo2} and $c_R$ by \eqref{eq:crgeneo2ACS}, we have proved the following spectral estimate:
\begin{equation}
	\label{eq:specgeneoACS2}
	c_T \le \lambda(M^{-1}_{GenEO2ACS}\,A) \le  c_R\,.		
\end{equation}
Constants $c_T$ and $c_R$ are stable with respect to $\epsilon_A$ and the spectrum of $E \tilde E^{-1}$ so that \eqref{eq:specgeneoACS2} proves the stability of preconditioner $M^{-1}_{GenEO2ACS}$~\eqref{eq:Mm1geneo2} w.r.t. inexact solves.

\begin{remark}
	\label{rk:geneo2notrobust}
	Had we taken the GenEO-2 algorithm introduced in \cite{haferssas:2016:ASM} and modified only the coarse space solves:
\begin{equation}
	\label{eq:Mm1geneo2notrobust}
  \widetilde M^{-1}_{GenEO,2} =  Z\, \tilde E^{-1}\,Z^T +   (\id - \tilde P_0)\, (\sum_{i=1}^N R_i^T\,D_i\,  B_i^{-1}\,\, D_i\,\,R_i)\, (\id - \tilde P_0^T)\,,
	\end{equation}
the estimate for the upper bound of the preconditioned system would be for arbitray $\delta>0$
	\[
	\lambda_{max} \le \max( 1+\delta\, ,\, k_0 \gamma + (1+\frac{1}{\delta})\epsilon_A^2 k_0 \max_{1\le i \le N} \lambda_{max} (B_i^{-1} \,D_i\,R_i\, A\, R_i^T  D_i)^2  )
	\]
and would depend on the product of $\epsilon_A$ with the largest eigenvalue of the local operators $B_i^{-1} \,D_i\,R_i\, A\, R_i^T  D_i$. This last term can be very large and we were not able to guarantee robustness with respect to  approximte coarse solves. 
\end{remark}

\begin{remark}
	\label{rq:NeuNeu}
  If for some subdomain $i$, $1\le i\le N$, $B_i=\tilde A_i$ and $\tilde A_i$ is symmetric positive semi-definite and $D_i R_i A R_i^T D_i$ is SPD, the eigenvalue problem \eqref{eq:eigAtildeB} will not contribute to the coarse space. 
  More precisely, the  contribution of the subdomain to the coarse space involves \eqref{eq:thCG:eig_geneoBNNnew} and will be $R_i^T\,D_i\,\ker (\tilde A_i)\bigoplus V^\gamma_{i,geneo}$. Also in Definition~\ref{def:precGenEO2ACS}, $B_i^\dag$ is the pseudo inverse of $B_i$ where $W_{i\gamma}$ is the image of $B_i$ which is orthogonal to $\ker (\tilde A_i)$.
\end{remark}


{\section{Annex} 
\label{sec:annex}
We explain here how to adapt the GenEO coarse space as defined in~\cite{Dolean:2015:IDDSiam} so that it will behave as the one defined in \cite{Spillane:2014:ASC}. In one sentence, it consists simply in computing the local components of the coarse space on the subdomain extended by one layer (or more). 

More precisely, we start from a domain decomposition $(\Omega_i)_{1\le i \le N}$ and inherited indices decomposition $({\mathcal N}_i)_{1\le i \le N}$ as defined in the present article. Let us denote with a tilde $\tilde {}$ all the quantities related to the subdomains $\Omega_{\tilde i}$ obtained by extending by one (or more) layers of cells subdomains $\Omega_i$, see fig.~\ref{fig:omegapmusonelayer}. Similarly to~\cite{Dolean:2015:IDDSiam}, we define  
\[
	{\mathcal N}_{\tilde i} := \{ k \in {\mathcal N}\ | \ meas(Supp(\phi_k)\cap \Omega_{\tilde i})>0\}\,.
\]
\begin{figure}
	\label{fig:omegapmusonelayer}
	\begin{center}
  \includegraphics[width=0.43\textwidth]{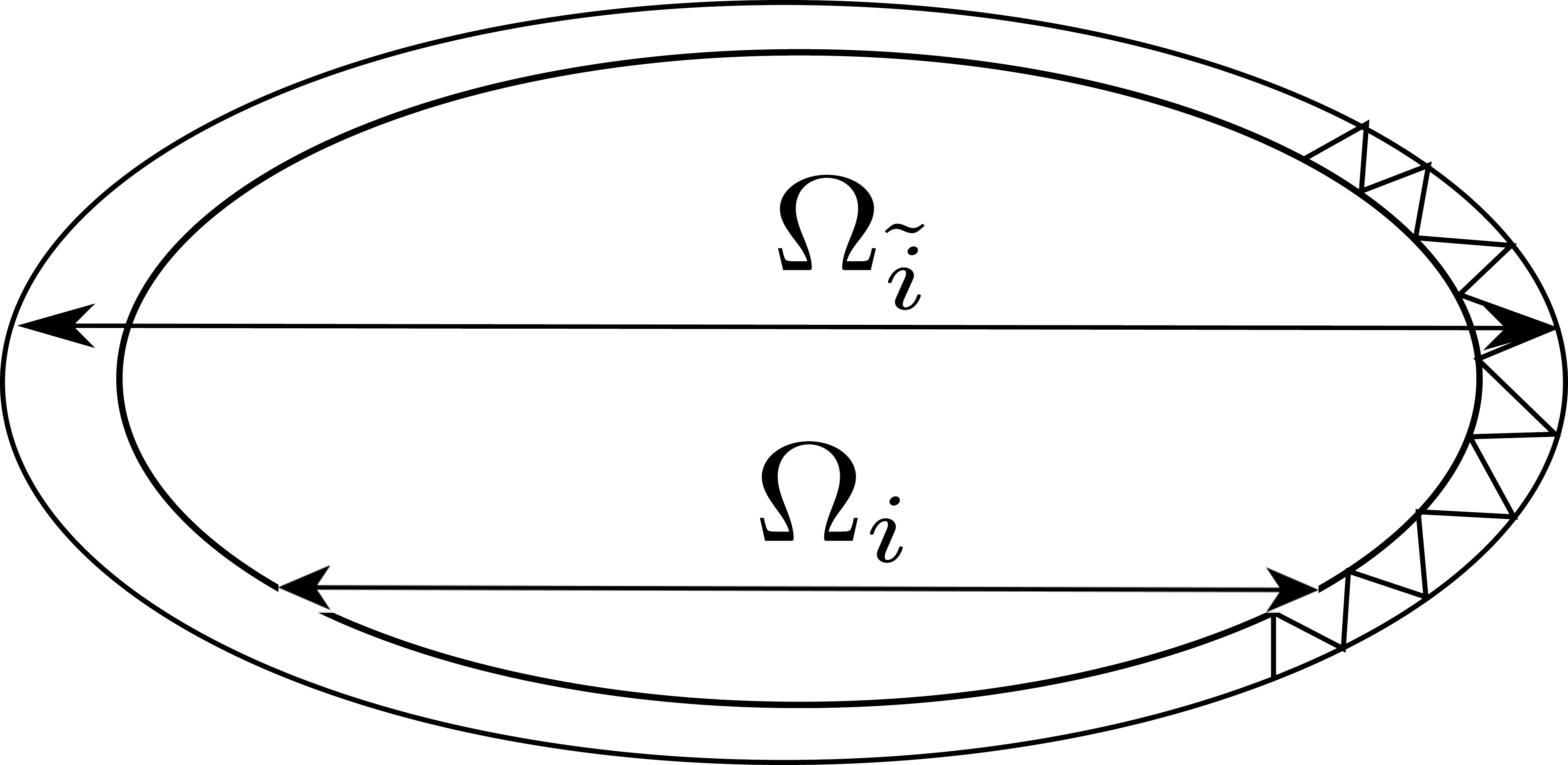} 
 \caption{ A subdomain and its extension by one layer}		
	\end{center}
\end{figure}
Since $\Omega_i \subset \Omega_{\tilde i}$ , we have 
\[
		{\mathcal N}_i \subset {\mathcal N}_{\tilde i} \,.
\]
 Also from the partition of unity on the original decomposition, we can define a partition of unity on $({\mathcal N}_{\tilde i})_{1\le i\le N}$ inherited from the one on $({\mathcal N}_i)_{1\le i\le N}$ by defining diagonal matrices $(D_{\tilde i})_{1\le i \le N}$ in the following manner:
 \[
 	(D_{\tilde i})_{kk} := \left\{
	\begin{array}{lcl}
		(D_i)_{kk} &\text{ if }& k\in {\mathcal N}_i \\
		0 &\text{ if }& k\in {\mathcal N}_{\tilde i} \backslash {\mathcal N}_i
	\end{array} 
	\right. \,.
 \]
We have clearly a partition of unity:
\[
	\id_{} = \sum_{i=1}^N R_{\tilde i}^T D_{\tilde i} R_{\tilde i}\,.
\]
Also since for all subdomains $1\le i\le N$, the entries of $D_{\tilde i}$ are zero on the added layers, we have the following equality:
\[
  D_i R_i = R_i R_{\tilde i}^T D_{\tilde i} R_{\tilde i}\,.
\]
The coarse space is built by first introducing the Neumann matrices $A_{\tilde i}^{Neu}$ on subdomains $\Omega_{\tilde i}$ for $1\le i \le N$, as in~\cite{Spillane:2014:ASC}, so that we have:
\[
	\sum_{i=1}^N (A_{\tilde i}^{Neu} R_{\tilde i}{\mathbf U}\,,\,R_{\tilde i}{\mathbf U}) \le {\widetilde\MC} (A {\mathbf U}\,,\,{\mathbf U})\,,
\]
where ${\widetilde\MC}$ is the maximum multiplicity of the intersections of subdomains $\Omega_{\tilde i}$. Let $V_{\tilde ik}$ be the eigenvectors of the following generalized eigenvalue problem:
\[
	D_{\tilde i}\, R_{\tilde i}\,A\,R_{\tilde i}^T\,D_{\tilde i}\,	V_{\tilde ik} = \lambda_{ik} A_{\tilde i}^{Neu}\, V_{\tilde ik}\,.
\]
Note that for $\lambda_{\tilde ik}$ not equal to $1$, the eigenvectors are harmonic for the interior degrees of freedom since for these points the left and right matrices have identical entries for the corresponding lines. Thus, for $\lambda_{\tilde ik} \neq 1$, we might as well zero the lines corresponding the interior degrees of freedom and keep only the entries of the degrees of freedom in the overlap. This GEVP is thus also of the type GenEO.

For a user-defined parameter $\tau$, let us define the coarse space as follows:
\[
	V_0 := Span\{ R_{\tilde i}^T\,D_{\tilde i}\,	V_{\tilde ik}\ |\ 1 \le i \le N,\ \lambda_{ik} > \tau \}\,,
\]
and a rectangular matrix $Z\in\R^{\#{\mathcal N}\times \#{\mathcal N}_0}$ whose columns are a basis of $V_0$ where ${\mathcal N}_0$ is a set of indices whose cardinal is the dimension of the vector space $V_0$. We also define local projections $(\pi_{\tilde i})_{1\le i \le N}$ on $Span\{ V_{\tilde ik}\ |\,\ \lambda_{ik} > \tau \}$ parallel to $Span\{ V_{\tilde ik}\ |\,\ \lambda_{ik} \le \tau \}$. 

We have then a stable decomposition. Indeed, let ${\mathbf U}\in\R^{\#{\mathcal N}}$,
\[
	{\mathbf U} = \sum_{i=1}^N R_i^T D_i R_i {\mathbf U} = \sum_{i=1}^N R_i^T R_i R_{\tilde i}^T D_{\tilde i} ( R_{\tilde i} {\mathbf U} - \pi_{\tilde i} R_{\tilde i} {\mathbf U}) + \sum_{i=1}^N R_i^T R_i R_{\tilde i}^T D_{\tilde i}  \pi_{\tilde i} R_{\tilde i} {\mathbf U}\,.  
\]
The last term is clearly in $V_0$ so that there exists ${\mathbf U}_0\in \R^{{\mathcal N}_0}$ such that 
\[
	Z {\mathbf U}_0 = \sum_{i=1}^N R_i^T R_i R_{\tilde i}^T D_{\tilde i}  \pi_{\tilde i} R_{\tilde i} {\mathbf U} = \sum_{i=1}^N R_{\tilde i}^T D_{\tilde i}  \pi_{\tilde i} R_{\tilde i} {\mathbf U} \in V_0\,.
\]
Let us define 
\[
	{\mathbf U}_i := R_i R_{\tilde i}^T D_{\tilde i} ( R_{\tilde i} {\mathbf U} - \pi_{\tilde i} R_{\tilde i} {\mathbf U}).
\]
This decomposition is stable since 
\[
\begin{split}
\sum_{i=1}^N (A R_i^T {\mathbf U}_i \,,\,R_i^T{\mathbf U}_i )
 &= \sum_{i=1}^N (A R_i^T R_i R_{\tilde i}^T D_{\tilde i} ( R_{\tilde i} {\mathbf U} - \pi_{\tilde i} R_{\tilde i} {\mathbf U}) \,,\, R_i^T R_i R_{\tilde i}^T D_{\tilde i} ( R_{\tilde i} {\mathbf U} - \pi_{\tilde i} R_{\tilde i} {\mathbf U}) ) \\
 &= \sum_{i=1}^N (A  R_{\tilde i}^T D_{\tilde i} ( R_{\tilde i} {\mathbf U} - \pi_{\tilde i} R_{\tilde i} {\mathbf U}) \,,\, R_{\tilde i}^T D_{\tilde i} ( R_{\tilde i} {\mathbf U} - \pi_{\tilde i} R_{\tilde i} {\mathbf U}) ) \\ 
 &\le \tau \sum_{i=1}^N (A_{\tilde i}^{Neu} R_{\tilde i} {\mathbf U} \,,\, R_{\tilde i} {\mathbf U} ) \le  \tau\widetilde\MC\, (A{\mathbf U} \,,\, {\mathbf U} )\,.
\end{split}
\]
Note also that Assumption~2.1 of \cite{Spillane:2014:ASC} is automatically satisfied in the finite element framework chosen here whereas Assumptions~3.12 and 3.13 are not needed here since our construction is simpler.}


	\bibliographystyle{plain}
	\bibliography{../jsm,../../elasticite/NotesdeCours/Poly/bookddm}

\begin{thebibliography}{10}

\bibitem{Blatt:2016:Distributed}
Markus Blatt, Ansgar Burchardt, Andreas Dedner, Christian Engwer, Jorrit
  Fahlke, Bernd Flemisch, Christoph Gersbacher, Carsten Gr{\"a}ser, Felix
  Gruber, Christoph Gr{\"u}ninger, et~al.
\newblock The distributed and unified numerics environment, version 2.4.
\newblock {\em Archive of Numerical Software}, 4(100):13--29, 2016.

\bibitem{Dolean:2015:IDDSiam}
Victorita Dolean, Pierre Jolivet, and Fr\'ed\'eric Nataf.
\newblock {\em An Introduction to Domain Decomposition Methods: algorithms,
  theory and parallel implementation}.
\newblock SIAM, 2015.

\bibitem{Dryja:1996:MSM}
Maksymilian Dryja, Marcus~V. Sarkis, and Olof~B. Widlund.
\newblock Multilevel {S}chwarz methods for elliptic problems with discontinuous
  coefficients in three dimensions.
\newblock {\em Numer. Math.}, 72(3):313--348, 1996.

\bibitem{Efendiev:2012:RDD}
Yalchin Efendiev, Juan Galvis, Raytcho Lazarov, and Joerg Willems.
\newblock Robust domain decomposition preconditioners for abstract symmetric
  positive definite bilinear forms.
\newblock {\em ESAIM Math. Model. Numer. Anal.}, 46(5):1175--1199, 2012.

\bibitem{Gower:2016:SBB}
R.~M. Gower, D.~Goldfarb, and P.~Richt\`arik.
\newblock Stochastic block bfgs: Squeezing more curvature out of data.
\newblock Technical report, Arxiv, 2016.

\bibitem{Griebel:1995:ATA}
M.~Griebel and P.~Oswald.
\newblock On the abstract theory of additive and multiplicative {S}chwarz
  algorithms.
\newblock {\em Numer. Math.}, 70(2):163--180, 1995.

\bibitem{haferssas:2016:ASM}
R.~Haferssas, P.~Jolivet, and F.~Nataf.
\newblock An {A}dditive {S}chwarz {M}ethod {T}ype {T}heory for {L}ions's
  {A}lgorithm and a {S}ymmetrized {O}ptimized {R}estricted {A}dditive {S}chwarz
  {M}ethod.
\newblock {\em SIAM J. Sci. Comput.}, 39(4):A1345--A1365, 2017.

\bibitem{Haferssas:2015:RCS}
Ryadh Haferssas, Pierre Jolivet, and Fr{\'e}d{\'e}ric Nataf.
\newblock A robust coarse space for optimized {S}chwarz methods:
  {SORAS}-{G}en{EO}-2.
\newblock {\em C. R. Math. Acad. Sci. Paris}, 353(10):959--963, 2015.

\bibitem{Hecht:2012:NDF}
F.~Hecht.
\newblock New development in {F}reefem++.
\newblock {\em J. Numer. Math.}, 20(3-4):251--265, 2012.

\bibitem{Jolivet:2014:HPD}
Pierre Jolivet and Fr\'ed\'eric Nataf.
\newblock Hpddm: {High-Performance Unified framework for Domain Decomposition
  methods, MPI-C++ library}.
\newblock {https://github.com/hpddm/hpddm}, 2014.

\bibitem{Lions:1990:SAM}
Pierre-Louis Lions.
\newblock On the {S}chwarz alternating method. {III:} a variant for
  nonoverlapping subdomains.
\newblock In Tony~F. Chan, Roland Glowinski, Jacques P{\'e}riaux, and Olof
  Widlund, editors, {\em Third International Symposium on Domain Decomposition
  Methods for Partial Differential Equations , held in Houston, Texas, March
  20-22, 1989}, Philadelphia, PA, 1990. SIAM.

\bibitem{Mandel:1992:BDD}
Jan Mandel.
\newblock Balancing domain decomposition.
\newblock {\em Comm. on Applied Numerical Methods}, 9:233--241, 1992.

\bibitem{Mandel:2008:MBD}
Jan Mandel, Bed{\v{r}}ich Soused{\'\i}k, and Clark~R. Dohrmann.
\newblock {\em On Multilevel BDDC}, pages 287--294.
\newblock Springer Berlin Heidelberg, Berlin, Heidelberg, 2008.

\bibitem{Nepomnyaschikh:1991:MTT}
Sergey~V. Nepomnyaschikh.
\newblock Mesh theorems of traces, normalizations of function traces and their
  inversions.
\newblock {\em Sov. J. Numer. Anal. Math. Modeling}, 6:1--25, 1991.

\bibitem{Nicolaides:DCG:1987}
Roy~A. Nicolaides.
\newblock Deflation of conjugate gradients with applications to boundary value
  problems.
\newblock {\em SIAM J. Numer. Anal.}, 24(2):355--365, 1987.

\bibitem{Pechstein:2017:UFA}
Clemens Pechstein and Clark~R. Dohrmann.
\newblock A unified framework for adaptive {BDDC}.
\newblock {\em Electron. Trans. Numer. Anal.}, 46:273--336, 2017.

\bibitem{Prudhomme:2006:DSE}
C.~Prud'homme.
\newblock A {D}omain {S}pecific {E}mbedded {L}anguage in c++ for automatic
  differentiation, projection, integration and variational formulations.
\newblock {\em Scientific Programming}, 14(2):81--110, 2006.

\bibitem{Schnabel:83:QNM}
R.~B. Schnabel.
\newblock Quasi-{N}ewton {M}ethods using {M}ultiple {S}ecant {E}quations.
\newblock Tech. rep. Computer Science Technical Reports 247-83, Paper 244,
  Colorado University-Computer Science, 1983.

\bibitem{Spillane:2014:ASC}
Nicole Spillane, Victorita Dolean, Patrice Hauret, Fr\'ed\'eric Nataf, Clemens
  Pechstein, and Robert Scheichl.
\newblock Abstract robust coarse spaces for systems of {PDE}s via generalized
  eigenproblems in the overlaps.
\newblock {\em Numer. Math.}, 126(4):741--770, 2014.

\bibitem{STCYR:2007:OMA}
Amik St-Cyr, Martin~J. Gander, and Stephen~J. Thomas.
\newblock Optimized {M}ultiplicative, {A}dditive, and {R}estricted {A}dditive
  {S}chwarz {P}reconditioning.
\newblock {\em SIAM J. Sci. Comput.}, 29(6):2402--2425 (electronic), 2007.

\bibitem{Tang:2009:CTL}
J.M. Tang, R.~Nabben, C.~Vuik, and Y.A. Erlangga.
\newblock Comparison of two-level preconditioners derived from deflation,
  domain decomposition and multigrid methods.
\newblock {\em Journal of Scientific Computing}, 39(3):340--370, 2009.

\bibitem{Toselli:2005:DDM}
Andrea Toselli and Olof Widlund.
\newblock {\em Domain Decomposition Methods - Algorithms and Theory}, volume~34
  of {\em Springer Series in Computational Mathematics}.
\newblock Springer, 2005.

\bibitem{Tu:2007:TLB}
Xuemin Tu.
\newblock Three-level {BDDC} in three dimensions.
\newblock {\em SIAM J. Sci. Comput.}, 29(4):1759--1780, 2007.

\bibitem{Tu:2011:TLB}
Xuemin Tu.
\newblock A three-level {BDDC} algorithm for a saddle point problem.
\newblock {\em Numer. Math.}, 119(1):189--217, 2011.

\bibitem{Zhang:1992:MSM}
Xuejun Zhang.
\newblock Multilevel {S}chwarz methods.
\newblock {\em Numer. Math.}, 63(4):521--539, 1992.

\end{thebibliography}

\end{document}